\documentclass[11pt, oneside]{amsart}
\usepackage[all,cmtip]{xy}
\usepackage[utf8x]{inputenc}
\usepackage{amsmath, amsthm, amssymb}
\usepackage{eucal, mathrsfs, yfonts}
\usepackage{mathtools}
\usepackage{amsfonts}
\usepackage{array}
\usepackage{hyperref}
\usepackage{graphicx}
\usepackage{caption}
\usepackage{pdfpages}
\usepackage{verbatim}
\usepackage{float}
\usepackage[english]{babel}
\usepackage{epstopdf}

\usepackage{geometry}                		% See geometry.pdf to learn the layout options. There are lots.
\newtheorem{thm}{Theorem}[section]
\theoremstyle{plain}

\newtheorem{prop}[thm]{Proposition}
\newtheorem{cor}[thm]{Corollary}

\theoremstyle{remark}

\numberwithin{equation}{section}

\newtheorem*{remark}{Remark}

\begin{document}

\title{Regularized Laplacian determinants of self-similar fractals}

\author[Joe P. Chen]{Joe P. Chen}
\address{Joe P. Chen\\
	Department of Mathematics\\
	Colgate Universty\\
	Hamilton, NY 13346\\
	USA.}
\email{jpchen@colgate.edu}

\author[Alexander Teplyaev]{Alexander Teplyaev}
\address{Alexander Teplyaev\\
	Department of Mathematics\\
	University of Connecticut\\
	Storrs, CT 06269\\
	USA.}
\email{teplyaev@math.uconn.edu}

\author[Konstantinos Tsougkas]{Konstantinos Tsougkas}

\address{Konstantinos Tsougkas\\
         Department of Mathematics\\
         Uppsala University\\
         751 05 Uppsala\\
         Sweden.}

\email{konstantinos.tsougkas@math.uu.se}

\maketitle

\begin{abstract}
	We study the spectral zeta functions of the Laplacian on 
			fractal sets which are locally self-similar fractafolds, in the sense of Strichartz.  These functions are known to meromorphically extend to the entire complex plane, and the locations of their poles, sometimes referred to as complex dimensions, are of special interest. We give examples of 
			locally self-similar sets such that their complex dimensions are not on the imaginary axis, which allows us to interpret their Laplacian determinant as the regularized product of their eigenvalues. We then investigate a connection between the logarithm of the determinant of the discrete graph Laplacian and the regularized one.

\end{abstract}

\section{Introduction} \noindent
If we have a differential Laplace operator $\mathcal{L}$ with discrete spectrum, we can define its spectral zeta function as
$$\zeta_{\mathcal{L}}(s)=\text{Tr}\left\{\frac{1}{\mathcal{L}^s}\right\}=\sum_n \frac{1}{\lambda_n^s},$$
where the zero eigenvalue is excluded and eigenvalues are added according to their multiplicities. Equivalently, if we have the heat kernel trace $K(t)=\sum_n e^{-\lambda_n t}$, we can define the spectral zeta function as the Mellin transform of the heat kernel trace minus one to remove the eigenvalue zero, namely:
$$\zeta_{\mathcal{L}}(s)=\frac{1}{\Gamma(s)}\int_0^{\infty} (K(t)-1) t^{s-1} dt.$$
Now we can write formally 
$$\det\mathcal{L}=\prod_{i=1}^{\infty} \lambda_i$$
to be the product of its nonzero eigenvalues, and we make the convention for the rest of this paper that the zero eigenvalue will always be excluded from any determinant. Of course, in the cases that we will consider the eigenvalues diverge to infinity, so this product exists only in a formal sense. We are interested, however, in assigning some meaning to it, which we can do by the following formal observations
$$\zeta_{\mathcal{L}}'(s)=\left( \sum_{i=1}^{\infty} \frac{1}{\lambda_i^s}\right)'=- \sum_{i=1}^{\infty} \frac{1}{\lambda_i^s} \log{\lambda_i}.$$
Evaluating at $s=0$ we get
$$ \zeta'_{\mathcal{L}}(0)=-\sum_{i=1}^{\infty} \log{\lambda_i}=\log{\prod_{i=1}^{\infty} \lambda_i}=-\log{\det\mathcal{L}},$$
so we can define the determinant of the operator $\mathcal{L}$ to be $\det\mathcal{L}=e^{-\zeta'_{\mathcal{L}}(0)}$.

The spectrum of the Laplace operator on fractals has been the focus of considerable work, see \emph{e.g.\@} \cite{chen2015spectral,chen2016singularly,fukushima1992spectral,kigami1993weyl,shima1991eigenvalue,shima1996eigenvalue,fractafold,strichartz2012exact,strichartz2012spectral}. Given a post-critically finite (p.c.f.) self-similar set (see \cite{kigami2001analysis} for the definition of p.c.f.\@), one can compute its spectral dimension $d_s$ and walk dimensions $d_w$, and these dimensions are connected via the Einstein relation $d_s=2\frac{d_f}{d_w}$ where $d_f$ is the Hausdorff dimension.
In \cite{grabner,teplstei,teplyaev2004spectral,teplyaev} the spectral zeta functions have been studied, and while they are defined initially only for $s>\frac{d_s}{2}$, they are shown to meromorphically extend to the entire complex plane. Their poles, also called \emph{complex dimensions} \cite{lapidus2000fractal}, are studied in \cite{teplstei} and it is proven that for a large class of p.c.f. fractals with symmetries, that the poles can only be on the imaginary axis or on the axis where ${\rm Re}(s)=\frac{d_s}{2}$.

Our long term motivation comes from quantum physics, in particular such recent papers as \cite{akkermans2013statistical,
	akkermans2012wave,
akkermans2009physical, 
	akkermans2010thermodynamics, akkermans2013spontaneous,
	dunne2012heat,
    elizalde2012introduction,
	lauscher2005fractal,
	reuter2011fractal,tanese2014fractal} and more classical works \cite{elizalde1994zeta,englert1987metric,hawking1977zeta,knizhnik1988fractal}. Our immediate mathematical motivation is twofold. On the one hand, it comes from the following statement found in \cite{grabner} and \cite{derfel2012laplace}:

``If there were no poles on the imaginary axis, then $e^{-\zeta'_{\Delta}(0)}$ would be the regularized product of eigenvalues or the Fredholm determinant of $\Delta$."

On the other hand, in \cite{karl} a connection between the determinant of the discrete Laplacians and the regularized determinant has been made in the setting of the discrete Euclidean torus. Specifically, let $N=(n_1(u), \dots n_d(u))$ denote a $d$-tuple of positive integers parametrized by $u \in \mathbb{Z}$, such that for each $j$, we have $\frac{n_j(u)}{u} \rightarrow a_j$ as $u \rightarrow \infty$. One then defines the $d$-dimensional discrete torus as the product space
$$DT_{N(u)}=\prod_{i=1}^d n_j(u)\mathbb{Z} / \mathbb{Z}.$$ If $A$ is the diagonal matrix with entries $a_j$ and $V(a)=a_1\cdots a_d$, the authors of \cite{karl} established the formula
\begin{align}
\label{asymp}
\log{\det\Delta_{DT_{N(u)}}}=V(N(u))\mathcal{I}_d(0)+\log{u^2}+\log{\det\Delta_{RT,A}}+o(1) \text{ as } u\rightarrow \infty,
\end{align}
where $RT,A$ is the real torus $A \mathbb{Z}^d / \mathbb{R}^d$, and $\mathcal{I}$ is a specific special function. A variation of this result was also studied in \cite{Vert}.

The goal of this paper is to give examples of fractals whose spectral zeta functions have no poles on the imaginary axis, which then allows us to define the corresponding Laplacian determinant, interpreted as the regularized product of the Laplace eigenvalues. 
%\begin{comment}
This result can be stated as a regularized limit and has been proven again with a different methodology in \cite{Vert}. More specifically, if $f \in C^{\infty}(\mathbb{R}^{+},\mathbb{C})$ and is of the form
	$$f(x)=\sum_{j=1}^{N-1}\sum_{k=0}^{M_j} a_{jk} x^{a_j}\log^k{x}+\sum_{k=0}^{M_0}a_{0k} \log^k{x}+o(x^{a_N}\log^{M_N}{x})$$
	for some $N \in \mathbb{N}$, $(a_j) \subset \mathbb{C}$ such that $({\rm Re}(a_j))$ is monotonically decreasing and ${\rm Re}(a_N)<0$, then we define the regularized limit of $f$ as
	$\widetilde{\lim_{x \rightarrow \infty}} f(x)=a_{00}.$
Then as in \cite{Vert} we can restate the above result as
$$\widetilde{\lim_{n \rightarrow \infty}} \log\det\Delta_n=\log\det\Delta.$$
%\end{comment}

Motivated by the above mentioned connection between a classical determinant and a zeta regularized determinant, we investigate a similar relation on some fractal examples. In this paper we study three concrete examples: the diamond fractal, the $N-1$ dimensional double Sierpi\'nski gaskets ($SG^N)$, and the double $pq$-model on the unit interval. All three examples satisfy spectral decimation, which leads to closed-form expressions for the spectral zeta functions and the Laplacian determinants. However, only for the double Sierpi\'nski gaskets and the double $pq$-model do we have exact analogs of \eqref{asymp}. Details will be described in subsequent sections, after a review of basic notions from analysis on fractals and graph theory.

\section{Notions of analysis on fractals and graph theory} \noindent 
The fractals we will study are self-similar sets defined in the following way. Given a compact connected metric space $(X,d)$, and injective contractions $F_i: X \rightarrow X$, $i \in \{1,2,..., m\}$, there exists a unique non-empty compact subset $K$ of $X$ that satisfies 
$$K=\bigcup_{i=1}^m F_i(K).$$
This will be our self-similar set. A fixed point $p_1$ of one of the maps $F_i$ for some $1\leq i\leq m$ is called an essential fixed point if there exists another fixed point $p_2$ such that $F_j(p_1)=F_k(p_2)$ for some $1 \leq j\neq k\leq m$.  Associated to $K$ is a sequence of approximating graphs $\{G_n : n\geq 0\}$, defined as follows. Let $V_0$ be the set consisting of the essential fixed points of the maps $F_i$, and $G_{0}$ be the complete graph on $V_{0}$. For $n\geq 1$, we define inductively
%and for all  $\omega\in W_n := \{1,2,\cdots, m\}^n$, and 
$$V_n:= \bigcup_{i=1}^m F_i(V_{n-1}),$$
%=\bigcup_{w \in W_n}V_w
%\;\; \; \text{ and }\; \; \; G_{n}:= \bigcup_{w \in W_n}G_w,$$
%where $G_w $ is the complete graph with vertices $V_w$ and  $F_{\omega}:=F_{w_1}\circ F_{w_{2}}\circ\cdots F_{w_n} $ for $\omega=w_1 w_2\cdots w_n$.
and declare $x,y\in V_n$ to be connected by an edge in $G_n$ (denoted $x\underset{n}{\sim} y$) if $F_i^{-1}(x)$ and $F_i^{-1}(y)$ are connected by an edge in $G_{n-1}$ for some $1\leq i\leq m$. We define the Dirichlet form on $G_n$ in the usual way
$$ \mathcal{E}_m(u,v)= \sum_{x\underset{m}{\sim} y} \left(u(x)-u(y) \right)\left(v(x)-v(y) \right), \quad u,v: V_m\to\mathbb{R}.$$
In many fractal examples it is possible to show that $\displaystyle \lim_{n\to\infty} r^{-n} \mathcal{E}_n(u,v)$ exists, where $r>0$ is a renormalization constant. We denote this limit by $\mathcal{E}(u,v)$, and write $\mathcal{E}(u) $ to stand for $\mathcal{E}(u,u)$. For example, the standard two-dimensional Sierpi\'nski gasket $SG_2$ satisfies this property with $r=\frac{3}{5}$. 
%The energy is
%$\mathcal{E}(u,v)=\lim \limits_{m \rightarrow \infty}{\mathcal{E}_m(u,v)}$ which is a bilinear quadratic form and for $u=v$ we simply denote $\mathcal{E}(u)$. 
Every function of finite energy is continuous. In fact $dom\mathcal{E}$, the space of functions of finite energy, is a dense subspace of the space of continuous functions on $K$. Given the energy, we can define the Laplacian, which is the main focus of our study.
For $u \in dom\mathcal{E}$ we say $u \in dom\Delta_{\mu}$ and $\Delta_{\mu}u=f$ if
$$\mathcal{E}(u,v)=- \int_K {fv} \mathrm{d\mu} \text{  for all } v\in dom_0 \mathcal{E} $$
where $dom_0 \mathcal{E}$ denotes the subset of $dom \mathcal{E}$ such that the functions also vanish on the boundary. We will use the convention from now on $-\Delta_{\mu}=\mathcal{L}$, we always assume that $\mu$ will be the standard self-similar measure, so it will be omitted from the notation.

\subsection{Spectral decimation and zeta functions}

The Laplace operator $\mathcal{L}$ with Neumann or Dirichlet conditions is a non-negative self-adjoint operator with compact resolvent. Its spectrum consists of discrete eigenvalues such that
$$0 \leqslant \lambda_1 < \lambda_2 \leqslant \lambda_3 \leqslant \dots$$
with $\lambda_n \rightarrow \infty$, and thus we can define its spectral zeta function as in the introduction. The key technique to studying the spectrum of the Laplacian on fractals is spectral decimation. Essentially, spectral decimation allows us to recursively obtain the eigenvalues of a given graph level approximation from knowledge of the previous graph approximation. In the end, taking a limit gives us the spectrum of the Laplace operator on the self-similar set. More rigorously, we say that we have spectral decimation if all eigenvalues of $\mathcal{L}$ are of the form 
$$- \lambda^m \lim_{n\rightarrow \infty} \lambda^n R^{(-n)}(w)$$
for $w \in A$, where $A$ is a finite set and $R$ is a rational function. For the limit to exist the elements of the preimages $R^{-(n)}(w)$ must be chosen appropriately. The value $m$ stands for the so-called generation of birth of the eigenvalue $w$ and is independent of $n$; more information may be found at \cite{grabner,fukushima1992spectral,strichartz2006differential,teplyaev}. In many cases this rational function turns out to be a polynomial. The quantity $\lambda$ is also known as the time-scaling factor.

Now, let $R(z)=a_dx^d+\dots+\lambda x$ be a polynomial with real coefficients and $d\geq 2$ which satisfies $R(0)=0$ and $R'(0)= \lambda > 1$. We denote as $\Phi$ the entire function which is a solution of the functional equation 
$$\Phi (\lambda z)=R(\Phi (z)) \text{ with } \Phi(0)=0, \Phi ' (0)=1.$$
We also define the so-called polynomial zeta functions
$$\zeta_{\Phi,w}(s)= \sum\limits_{\substack{\Phi(-\mu)=w \\ \mu >0}} \mu^{-s}$$
or equivalently as
$$\zeta_{\Phi,w}(s)= \lim_{n \rightarrow \infty} \sum_{z \in R^{-n}(w)} (\lambda^n z)^{-s}.$$
These zeta functions have been studied in \cite{grabner}, \cite{teplyaev} and were used in the meromorphic extension of the spectral zeta functions. We know the following facts about them.
For $w<0$ we have that 
$$\zeta_{\Phi,w}(0)=0 \quad \text{ and } \quad \zeta'_{\Phi,w}(0)=-\frac{\log{a_d}}{d-1}-\log{(-w)}$$
and for $w=0$
$$\zeta_{\Phi,0}(0)=-1 \quad \text{ and } \quad \zeta'_{\Phi,0}(0)=-\frac{\log{a_d}}{d-1}.$$
Moreover they can be meromorphically extended to the entire complex plane and have no poles on the imaginary axis. More specifically, all the poles are simple and are located on the imaginary line where ${\rm Re}(s)=\frac{\log{d}}{\log{\lambda}}$. For further information we refer the reader to \cite{grabner}, \cite{teplyaev}.

\subsection{Counting spanning trees in fractal graphs}

In the graph theoretic setting, the determinant of the discrete Laplacian is widely studied, as it is related to the enumeration of spanning trees via Kirchhoff's Matrix-Tree theorem. To be concrete, we define the \emph{combinatorial graph Laplacian} of a graph $G_n$ as $\Delta_n=D-A$, and the \emph{probabilistic graph Laplacian} as $\mathcal{L}_n=I-D^{-1}A$, where $D$ is the diagonal degree matrix and $A$ is the adjacency matrix. Then the number of spanning trees in $G_n$ may be expressed in either of two ways:
$$\tau(G_n)=\frac{\det\Delta_n}{|V_n|} \quad \text{ or }\quad \tau(G_n)=\left(\frac{\prod_i d_i}{\sum_i d_i}\right) \det\mathcal{L}_n,$$ where $d_i$ are the vertex degrees. One can further introduce the \emph{asymptotic complexity constant} of $(G_n)$, studied in \cite{LY05},
$$c=\lim_{n \rightarrow \infty} \frac{\log{\tau(G_n)}}{|V_n|}$$
provided that the limit exists.

For fractal graphs admitting spectral decimation, the determinant of the graph Laplacians, as well as the asymptotic complexity constant, has been evaluated in \cite{anema-tsougkas2016}. The key insight is that one can split the eigenvalues into two disjoint finite sets $A$ and $B$. If the rational function associated with spectral decimation is of the form $R(z)=\frac{P(z)}{Q(z)}$ with degree $d$, and $P_d$ is the leading coefficient of $P$, then
\begin{equation} {\label{eq:discretedet}}
\det\mathcal{L}_n=\left(\prod_{\alpha\in A}\alpha^{\alpha_n}\right)\left[\prod_{\beta\in B}\left(\beta^{\sum_ {k=0}^{n}{\beta_n^k}}\left(\frac{-Q(0)}{P_d}\right)^{\sum_{k=0}^n\beta_n^k\left(\frac{d^k-1}{d-1}\right)}\right)\right].
\end{equation}
where $\alpha_n=mult_n(\alpha)$ and $\beta_n^k=mult_nR^{-k}(\beta)$, where $mult_n(\lambda)$ denotes the multiplicity of the eigenvalue $\lambda$ of $\mathcal{L}_n$.
We refer the reader to \cite{anema-tsougkas2016}, \cite{bajorin2008vibration} and \cite{malozemov2003self} for details.

\section{Zeta function of the Diamond fractal}  \noindent

\begin{figure}[H]
\centering
\includegraphics[scale=0.3]{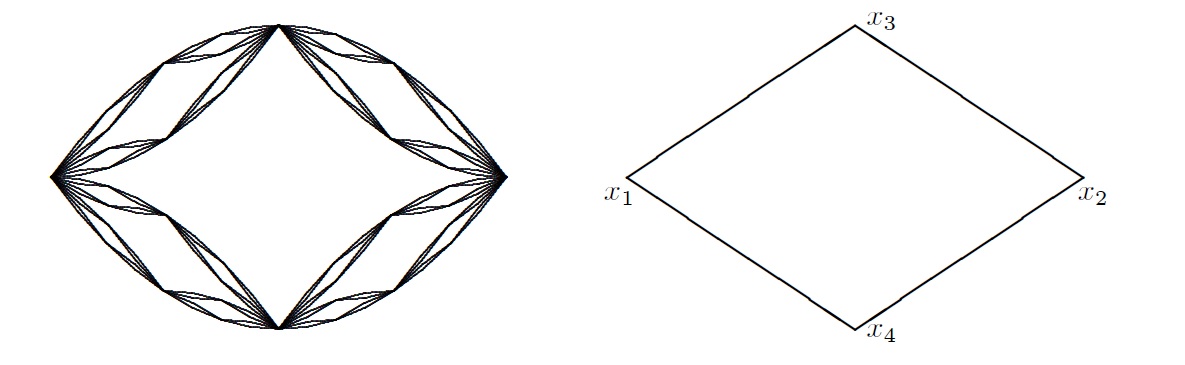}
\caption{The Diamond fractal and its $G_1$ approximating graph.}
\label{fig:diamond}
\end{figure}

The Diamond fractal has been recently studied due to its connections related to physics. In this section we show that it is a fractal with spectral zeta function such that it has no poles on the imaginary axis. Moreover, we establish a connection between the discrete and continuous determinants of the Laplacian bearing some resemblance to \cite{karl}. 
\begin{prop}
The spectral zeta function of the Diamond fractal factorizes as follows
$$\zeta_{\mathcal{L}}(s)=\frac{4^s(4^s-1)}{6}\left( \frac{4}{4^s-4}+\frac{2}{4^s-1}\right)\zeta_{\Phi,0}(s)$$
and thus has no poles on the imaginary axis. Its regularized determinant is 
$$\det\mathcal{L}=2^{\frac{5}{9}}$$
\end{prop}
\begin{proof}
For the Diamond fractal spectral decimation has been done in \cite{bajorin2008vibration}, and it was obtained that 
$$ R(z)=2z(2+z) \text{ and } \lambda = R'(0)=4.$$
Then $\Phi$ satisfies the functional equation $\Phi(\lambda z)=R(\Phi (z))$, and thus we have that $\Phi(4z)=2\Phi(z)(2+\Phi(z))$. This allows us to say that
$$ \Phi(z)=-1 \Leftrightarrow \Phi(4z)=-2$$
and
$$ \Phi(z)=-2 \Leftrightarrow \Phi(4z)=0 \text{ and } \Phi(z) \neq 0.$$
We have that every eigenvalue of $\mathcal{L}$ is of the form $-4^m \lim\limits_{n\rightarrow \infty} 4^nR^{-n}(-1)$ and $mult_n(1)=\frac{4^n+2}{3}$. Thus we get that
$$\zeta_{\mathcal{L}}(s)= \sum_{n=1}^{\infty} \left(\frac{4^n+2}{3}\right) 4^{-ns} \zeta_{\Phi,-1}(s)=\frac{1}{3}\left( \frac{4}{4^s-4}+\frac{2}{4^s-1}\right)\zeta_{\Phi,-1}(s).$$
We will show that in fact we have no poles on the imaginary axis because of cancellations from the solutions of $4^s=1$.
By using the observations above, we obtain that
\begin{equation*}
\begin{split}
\zeta_{\Phi,-1}(s)=& \sum\limits_{\substack{\Phi(-\mu)=-1 \\ \mu >0}} \mu^{-s}=\frac{1}{2}\sum\limits_{\substack{\Phi(-4\mu)=-2 \\ \mu >0}} \mu^{-s}= \frac{1}{2}4^s \sum\limits_{\substack{\Phi(-4\mu)=-2 \\ \mu >0}} (4\mu)^{-s}=\frac{1}{2} 4^s \zeta_{\Phi,-2}(s) \\
\end{split}
\end{equation*} 
and
\begin{equation*}
\begin{split}
\zeta_{\Phi,-2}(s)=& \sum\limits_{\substack{\Phi(-\mu)=-2 \\ \mu >0}} \mu^{-s}=\sum\limits_{\substack{\Phi(-4\mu)=0 \\ \Phi(-\mu) \neq 0 \\\mu >0}} \mu^{-s}=  \sum\limits_{\substack{\Phi(-4\mu)=0 \\ \mu >0}} \mu^{-s}-\sum\limits_{\substack{\Phi(-\mu)=0 \\ \mu >0}} \mu^{-s}\\
&=4^s\sum\limits_{\substack{\Phi(-4\mu)=0 \\ \mu >0}} (4\mu)^{-s}-\sum\limits_{\substack{\Phi(-\mu)=0 \\ \mu >0}} \mu^{-s}= (4^s-1) \zeta_{\Phi,0}(s) 
\end{split}
\end{equation*} 
From this it follows that
$$\zeta_{\mathcal{L}}(s)=\frac{4^s(4^s-1)}{6}\left( \frac{4}{4^s-4}+\frac{2}{4^s-1}\right)\zeta_{\Phi,0}(s)$$
which proves that there are no poles on the imaginary axis.

Now by differentiating and using the fact that $\zeta_{\Phi ,0}(0)=-1$ and that $\zeta'_{\Phi , 0}(0)=-\frac{\log{a_d}}{d-1}=-\log{2}$, we obtain $\zeta'_{\mathcal{L}}(0)=-\frac{5}{9}\log{2}$. But due to the discussion in the introduction, this essentially means that the absence of poles on the imaginary axis allows us to interpret this as the regularized product of the eigenvalues, and thus
$$\det\mathcal{L}=e^{-\zeta'_{\mathcal{L}}(0)}=2^{\frac{5}{9}}.$$

\end{proof}

\begin{remark}
At first glance the complex dimensions would be located at the positions such that $4^s=4$ which are $s=1+\frac{ik\pi}{\log{2}}$  and at ${\rm Re}(s)=\frac{\log{2}}{\log{4}}=\frac{1}{2}$ due to the poles of the polynomial zeta functions. However, it was shown in \cite{teplstei} that the the complex dimensions can only be on the imaginary axis, which we have proven is not the case, and at ${\rm Re}(s)=\frac{d_s}{2}=1$ and thus, we can deduce that all the poles of the polynomial zeta functions must be canceled by the zeros of the geometric part, which we observe that is indeed the case for ${\rm Re}(s)=\frac{1}{2}$.
\end{remark}
\begin{remark}
The value $\log{2}$ can also be interpreted as the tree entropy or the asymptotic complexity constant of the sequence of the fractal graphs approximating the Diamond fractal. Thus $\log{\det\mathcal{L}}=\frac{5}{9}c$.
\end{remark}

For the diamond fractal it has been calculated in \cite{anema-tsougkas2016} that $\det\mathcal{L}_n=2^{-\frac{1}{9}(2\cdot 4^n-6n-11)}$. Then using the fact that $\log{\det\mathcal{L}}=\frac{5}{9}\log{2}$ we conclude that
$$\log{\det\mathcal{L}_n}=\frac{1}{5}(-2\cdot 4^n +6n+11) \log{\det\mathcal{L}}$$
We can see that the regularized determinant does not appear as an exact constant as in \cite{karl}, despite the fact that there are no poles on the imaginary axis. The resemblance that appears here must be attributed to numerical coincidence.

In \cite{dunne2012heat} there is also a formula involving the Riemann zeta function for the diamond fractals with spectral zeta function 
$$\zeta_D(s)=\frac{\zeta_R(2s)}{\pi^{2s}}l^{d_f-1}\left( \frac{1-l^{1-d_ws}}{1-l^{d_f-d_ws}}  \right)$$
where $\zeta_R$ is the Riemann zeta function, $l$ is a side length constant, and $d_w$ and $d_f$ are the walk and Hausdorff dimensions correspondingly.

\section{Zeta function of double Sierpi\'nski gaskets} \noindent

The spectral zeta functions for Dirichlet and Neumann boundary conditions for the standard self-similar Laplacian on the standard two dimensional Sierpi\'nski gasket have been calculated in \cite{grabner}, \cite{teplyaev} as follows:
\begin{equation*}
\begin{split}
\zeta_{\mathcal{L}}^D(s)=&5^{-s}\zeta_{\Phi ,-2}(s)+\left( \frac{3}{2(5^s-3)}-\frac{3}{2(5^s-1)} \right)5^{-s}\zeta_{\Phi ,-3}(s) \\ 
& +\left( \frac{1}{2(5^s-3)}+\frac{3}{2(5^s-1)} \right)\zeta_{\Phi ,-5}(s)
\end{split}
\end{equation*}
and
$$\zeta_{\mathcal{L}}^N(s)=\left( \frac{1}{2(5^s-3)}+\frac{3}{2(5^s-1)} \right)\zeta_{\Phi ,-3}(s)
 +\left( \frac{3\cdot 5^{-s} }{2(5^s-3)}-\frac{5^{-s}}{2(5^s-1)} \right)\zeta_{\Phi ,-5}(s).$$
 
Notice that the poles on the imaginary axis appear to be at the points such that $5^s=1$. However, some are canceled out by the observation in \cite{grabner} that $\zeta_{\Phi , -5}(s)=(5^s-1) \zeta_{\Phi ,0}(s)$. Unfortunately a similar argument cannot work for $\zeta_{\Phi ,-3}(s)$, and numerical calculations by the authors in \cite{grabner} indicate that we indeed have poles at $5^s=1$.

In \cite{fractafold} the double Sierpi\'nski gasket was defined. Essentially, it is the fractal created by taking two copies of the regular Sierpi\'nski gasket and gluing them at the boundary. Then it becomes a fractal without boundary, and its graph approximations are $4$-regular graphs.

\begin{figure}[H]
\centering
\includegraphics[scale=0.5]{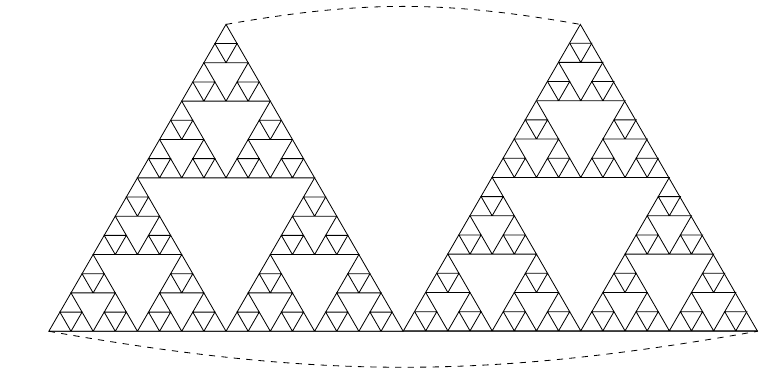}
\caption{The approximating graph $G_4$ of the double Sierpi\'nski gasket. Points connected by a dashed line are identified.}
\label{fig:dsg}
\end{figure}

We can also consider higher dimensional analogues of the Sierpi\'nski gasket. Denote $SG^N$ to be the $N-1$ dimensional Sierpi\'nski gaskets as in \cite{fukushima1992spectral}. The time-scaling factor is then $\lambda=N+2$. The Dirichlet spectral zeta function is evaluated in \cite{grabner} 
\begin{equation*}
\begin{split}
\zeta_{\mathcal{L}}^D(s)&=\lambda^{-s}\zeta_{\Phi,-2}(s)+\left( \frac{(N-1)^2-1}{2(\lambda^{s}-N)}-\frac{N}{2(\lambda^{s}-1)} \right)\lambda^{-s} \zeta_{\Phi, -N}(s)\\
&+\left( \frac{N-2}{2(\lambda^{s}-N)}+\frac{N}{2(\lambda^{s}-1)} \right) \zeta_{\Phi,-(N+2)}(s)
\end{split}
\end{equation*}
and by using explicit knowledge of the Neumann spectrum from \cite{bajorin2008vibration},\cite{fukushima1992spectral} we can also compute the Neumann spectral zeta function to be
\begin{equation*}
\begin{split}
\zeta_{\mathcal{L}}^N(s)&=\left(\frac{N-2}{2(\lambda^s-N)}+\frac{N}{2(\lambda^s-1)}\right)\zeta_{\Phi, -N}(s)\\
&+\left( \frac{N(N-2)}{2(\lambda^s-N)}-\frac{N-2}{2(\lambda^s-1)}\right)\lambda^{-s}\zeta_{\Phi,-(N+2)}(s).
\end{split}
\end{equation*}

We can now create the double $SG^N$ by taking two copies of $SG^N$ and gluing them together at the respective boundary points, making the appropriate $N$ identifications. Then it becomes a fractal without boundary, and the spectrum of the Laplace operator is the union of the Dirichlet and Neumann spectra with added multiplicities.

\begin{prop}

The spectral zeta function of the double $N-1$ dimensional Sierpi\'nski gasket has no poles on the imaginary axis. Its regularized determinant is

$$\det\mathcal{L}=\frac{2N^{\frac{1}{N-1}}}{(N+2)^{\frac{N-2}{N-1}}}.$$

\end{prop}

\begin{proof}

The spectral zeta function for the double $SG^N$ is the sum of the Dirichlet and Neumann spectral zeta functions of the single gaskets, and thus becomes
\begin{equation*}
\zeta_{\mathcal{L}} (s) = \lambda^{-s} \zeta_{\Phi,-2}(s)+ \frac{N(\lambda^s-1)-\lambda^s}{\lambda^s(\lambda^s-N)}\zeta_{\Phi,-N}(s)
+\frac{(N-1)(\lambda^s-2)}{(\lambda^s-1)(\lambda^s-N)}\zeta_{\Phi,-(N+2)}(s).
\end{equation*}
But since as in \cite{grabner} we have that $\zeta_{\Phi,-(N+2)}(s)=(\lambda^s-1)\zeta_{\Phi,0}(s)$, we see that we don't have any poles on the imaginary axis, which allows us to have the interpretation of a regularized determinant. By differentiating the formula above and taking into account that the spectral decimation function is $R(z)=z(N+2+z)$ with $d=2$, $a_d=1$ and also that 
\begin{equation*}
\begin{split}
&\zeta_{\Phi,0}(0)=-1 \hspace{0.695cm} \text{ and } \quad \zeta'_{\Phi,0}(0)=0\\
&\zeta_{\Phi,-N}(0)=0 \hspace{0.37cm} \text{ and } \quad \zeta'_{\Phi,-N}(0)=-\log{N}\\
&\zeta_{\Phi,-2}(0)=0  \hspace{0.48cm} \text{ and } \quad \zeta'_{\Phi,-2}(0)=-\log{2}
\end{split}
\end{equation*}
we obtain that
$$\zeta'(0)=-\log{2}-\frac{1}{N-1}\log{N}+\frac{N-2}{N-1}\log{(N+2)}$$
and therefore
\begin{equation}{\label{eq:logdetN}}
\log{\det\mathcal{L}}=\log{2}+\frac{1}{N-1}\log{N}-\frac{N-2}{N-1}\log{(N+2)}
\end{equation}
from which the result follows.

\end{proof}
\begin{remark}
As in the case of the Diamond fractal, the zeros of the geometric part cancel all the poles of the polynomial zeta functions, and the only poles that remain are at ${\rm Re}(s)=\frac{d_s}{2}=\frac{\log{N}}{\log{(N+2)}}$.
\end{remark}

We establish now a result analogous to \cite{karl}.
\begin{cor}
For the discrete combinatorial graph Laplacian determinant of the double $SG^N$, we have that
$$\log{\det\Delta_n}=c|V_n|+n\log{(N+2)}+ \log{\det\mathcal{L}},$$
where $c$ is the asymptotic complexity constant which is
$$c=\frac{N-2}{N}\log{2}+\frac{N-2}{N-1}\log{N}+\frac{N-2}{N(N-1)}\log{(N+2)}.$$

\end{cor}

\begin{proof}

By using \eqref{eq:discretedet}, the fact that the spectrum is the union of the Dirichlet and Neumann spectra with added multiplicities, and the eigenvalue multiplicities computed at \cite{fukushima1992spectral} and \cite{shima1991eigenvalue}, we can evaluate that 
$$\det\Delta_n=2^{(N-2)N^n+1}\cdot N^{\frac{(N-2)N^{n+1}+1}{N-1}}\cdot (N+2)^{\frac{(N-2)N^n+n(N-1)-N+2}{N-1}}$$
By using Kirchoff's Matrix-Tree theorem and also the fact that the number of vertices for the double Sierpi\'nski gasket graphs is $|V_n|=N^{n+1}$, we get that the number of spanning trees is 
$$\tau(G_n)=2^{(N-2)N^n+1}\cdot N^{\frac{(N-2)N^{n+1}-(N-1)(n+1)+1}{N-1}}\cdot (N+2)^{\frac{(N-2)N^n+n(N-1)-N+2}{N-1}}$$
and therefore the asymptotic complexity constant is 
$$c=\frac{N-2}{N}\log{2}+\frac{N-2}{N-1}\log{N}+\frac{N-2}{N(N-1)}\log{(N+2)}.$$
Using \eqref{eq:logdetN} we obtain the result.

\end{proof}

\begin{remark}
By this formula we can see the connection between different ``discrete and continuous" determinants. In fact, the asymptotic complexity constant can also be interpreted as a determinant, namely, a Fuglede--Kadison determinant. We refer the reader to \cite{fugledekadison},\cite{LY05} for more details. We then have a connection between ``discrete and continuous" determinants of the form
$$\log{\det\Delta_n}=\log{\text{Det}\Delta}\,|V_n|+n\log{(N+2)}+ \log{\det\mathcal{L}},$$
where $\text{Det}\Delta$ is the Fuglede--Kadison determinant.

\end{remark}

\begin{remark}
By using Kirchoff's Matrix-Tree theorem and the above calculations, we can also calculate the number of spanning trees for the single $N-1$ dimensional Sierpi\'nski gasket, confirming the formula conjectured in \cite{CCY07} and first proven via a different methodology in \cite{TeWa11}. The asymptotic complexity constant for the single and double pre-fractal Sierpi\'nski graphs are the same.

\end{remark}

\section{Zeta function of the double pq-model on the unit interval} \noindent

In \cite{strichartz2006differential} the unit interval can be realized as a p.c.f. self-similar set with two contractions. Then the standard self-similar measure is the Lebesgue measure, and the fractal Laplacian coincides with the standard $-\frac{d^2}{dx^2}$ operator. However, in \cite{teplyaev} a different fractal Laplacian on the unit interval has been constructed.
Let $0<p<1$ and $q=1-p$. Define contraction factors
$$r_1=r_3=\frac{p}{1+p} \quad \text{ and } \quad r_2=\frac{q}{1+p}$$
and measure weights
$$m_1=m_3=\frac{q}{1+q} \quad \text{ and } \quad m_2=\frac{p}{1+q}$$
and observe that $m_1+m_2+m_3=r_1+r_2+r_3=1$. We define the contractions $F_i:\mathbb{R} \rightarrow \mathbb{R}$ for $i=1,2,3$ as $F_i(x)=r_ix+(1-r_i)p_i$ where $p_i$ is $0,\frac{1}{2},1$, respectively, or equivalently the fixed point of $F_i$. Then the unit interval is the self-similar set created by these contractions, and as usual $V_n=\bigcup F_i (V_{n-1})$ with the boundary being $V_0=\{0,1\}$. As our self-similar probability measure we take the unique measure satisfying $\mu=\sum_{j=1}^3 m_j \mu\circ F_j$. Then we have that
$$\Delta_{\mu}(x)=\lim_{n \rightarrow \infty} (1+\frac{2}{pq})^n \Delta_n f(x)$$
where the discrete graph Laplacians are
$$\Delta_nf(x_k)=
\left\{
	\begin{array}{ll}
		2pf(x_{k-1})+2qf(x_{k+1})-2f(x_k) \\
     \quad \qquad \qquad   \text{or}\\
		2qf(x_{k-1})+2pf(x_{k+1})-2f(x_k)
	\end{array}
\right.$$
This Laplacian corresponds to a random walk as in Figure 3.
\begin{figure}
\begin{center}
\includegraphics[scale=0.85]{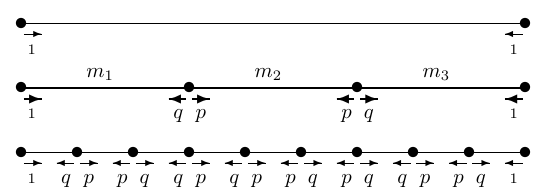}
\caption{The associated random walk of the $pq$-model.}
\end{center}
\end{figure}

Spectral decimation has been carried in \cite{teplyaev} for the Neumann case with rational function $R_p(z)=\frac{1}{pq}z(\frac{z^2}{4}+\frac{3z}{2}+2+pq)$ and the following is obtained.
\begin{prop}
$$\sigma(\Delta_{p,n})=\{0,-4\}\bigcup_{m=0}^{n-1} R_p^{-m}\{-2\pm 2q\}$$
and if $p \neq \frac{1}{2}$ then $d_s=\frac{\log{9}}{\log{(1+\frac{2}{pq})}}<1$ and
$$\zeta_{\mathcal{L}_{\mu}}^N(s)=\frac{1}{\lambda^{s}-1}(\zeta_{\Phi, w_1}(s)+\zeta_{\Phi, w_2}(s))$$ for $\lambda=1+\frac{2}{pq}$ and $w_1,w_2=-2 \pm 2q$.
\end{prop}
We can easily see that $R^{-1}(0)=\{0,-2-2p,-2-2q\}$ and $R^{-1}(-4)=\{-4,-2+2p,-2+2q\}$, and thus the spectrum is obtained as follows
\newline
\begin{center}
\includegraphics[scale=0.45]{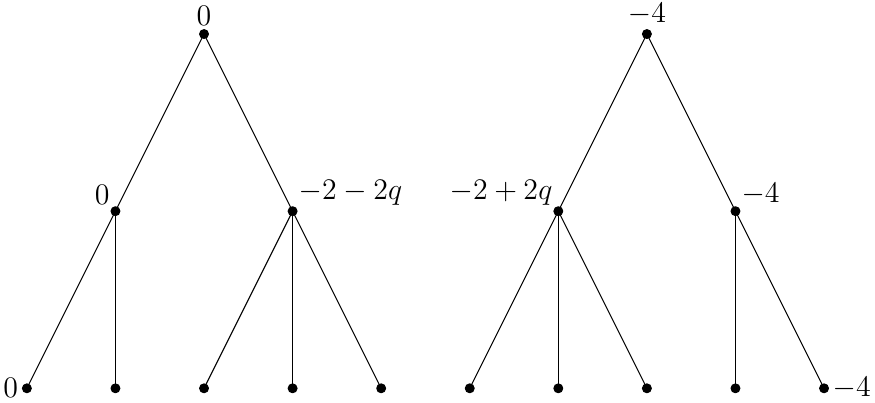}
\end{center}
This calculation for the spectral zeta function was omitted in \cite{teplyaev}. We clarify it here and also fix a typo in the formula.
Both eigenvalues $-2+2q$ and $-2-2q$ appear with multiplicity $1$ for $m\geq 1$,
and thus
$$\zeta_{\mathcal{L}}(s)=\sum_{m=1}^{\infty}\lambda^{-ms} \zeta_{\Phi, -2-2q}(s)+\sum_{m=1}^{\infty}\lambda^{-ms} \zeta_{\Phi, -2+2q}(s).$$
$$=\frac{\lambda^{-s}}{1-\lambda^{-s}}(\zeta_{\Phi, -2-2q}(s)+\zeta_{\Phi, -2+2q}(s))$$
Now, we mimic the construction of the double Sierpi\'nski gasket and glue this model with a copy of itself at the two boundary points. Its spectrum is again the union of the Dirichlet and Neumann spectra of the single case, and we get the following.
\begin{prop}
The spectral zeta function for the double $pq$-model is given by
$$\zeta_{\mathcal{L}_{\mu}}(s)=\zeta_{\Phi, 0}(s)+\zeta_{\Phi,-4}(s)$$
and thus it has no poles on the imaginary axis. Its regularized determinant is
$$\det\mathcal{L}_{\mu}=\frac{1}{pq}.$$
\end{prop}
Before we give this proof, we must calculate the Dirichlet spectrum for the single $pq$-model. By solving the Dirichlet eigenvalue equation on the first level we see that the eigenvalues are $-1-p$ and $-1+p$. These eigenvalues are initial and they show up at every level, and we encounter no exceptional eigenvalues by taking their preiterates. This means that the spectrum is of the following form.
\newline
\begin{center}
\includegraphics[scale=0.3]{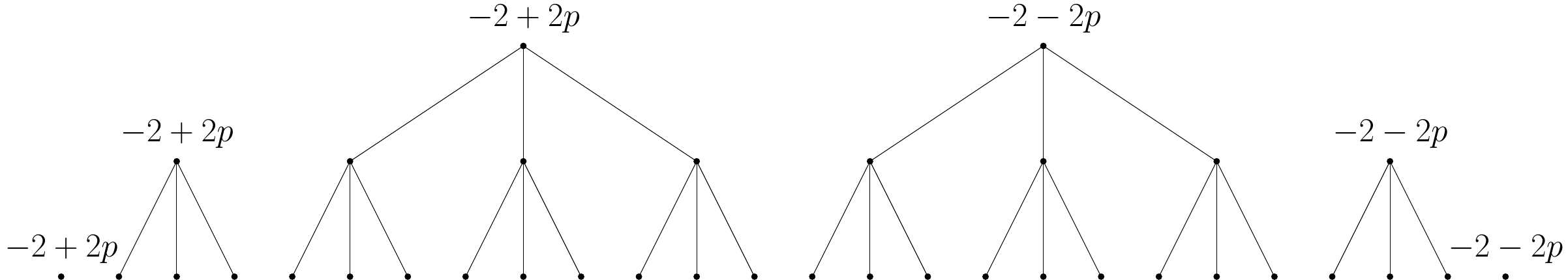}
\end{center}
We have that $\dim\Delta_n=3\dim\Delta_{n-1}-2$ for the Neumann spectrum, and $\dim\Delta_n=3\dim\Delta_{n-1}+2$ for the Dirichlet spectrum. Then the proof of the proposition goes as follows.
\begin{proof}
As in the case of the double Siepri\'nski gasket, it suffices to add the spectral zeta functions of the Neumann and Dirichlet spectrum. Every eigenvalue has multiplicity one, and from the calculations above, we have that the Dirichlet spectral zeta function is
$$\zeta_{\mathcal{L}_{\mu}}(s)=\frac{1}{\lambda^{s}-1}(\zeta_{\Phi, w_3}(s)+\zeta_{\Phi, w_4}(s))$$ where $w_3,w_4=-2 \pm 2p$. Since $\Phi(\lambda z)=R(\Phi(z))$ we can observe that
$$ \Phi(z)=-2-2q \Leftrightarrow \Phi(\lambda z)=0 \enspace \text{ and } \enspace \Phi(z) \neq 0 \enspace \text{ and } \enspace \Phi(z) \neq -2-2p$$
and
$$ \Phi(z)=-2+2q \Leftrightarrow \Phi(\lambda z)=-4 \enspace \text{ and } \enspace \Phi(z) \neq -4 \enspace \text{ and } \enspace \Phi(z) \neq -2+2p .$$
Then, as in the case of the Diamond fractal, we have that
\begin{equation*}
\begin{split}
\zeta_{\Phi,-2-2q}(s)&= \sum_{\substack{\Phi(-\mu)=-2-2q \\ \mu >0}} \mu^{-s}=\quad \sum_{\mathclap{\substack{\Phi(-\lambda\mu)=0 \\ \Phi(-\mu) \neq 0 \\ \Phi(-\mu) \neq -2-2p \text{ and } \mu >0}}} \mu^{-s}\\
&=  \sum\limits_{\substack{\Phi(-\lambda\mu)=0 \\ \mu >0}} \mu^{-s}-\sum\limits_{\substack{\Phi(-\mu)=0 \\ \mu >0}} \mu^{-s} \quad -\sum\limits_{\substack{\Phi(-\mu)=-2-2p \\ \mu >0}} \mu^{-s}\\
&= (\lambda^s-1) \zeta_{\Phi,0}(s)-\zeta_{\Phi, -2-2p}(s) 
\end{split}
\end{equation*} 
and similarly we have that
$$\zeta_{\Phi, -2+2q}(s))=(\lambda^s-1)\zeta_{\Phi ,-4}(s)-\zeta_{\Phi, -2+2p}(s).$$
The result then is obtained by adding the Dirichlet and Neumann spectral zeta functions and the fact that $\zeta'_{\Phi,0}(0)=-\frac{\log{\frac{1}{4pq}}}{2}$ and $\zeta'_{\Phi,-4}(0)=-\frac{\log{\frac{1}{4pq}}}{2}-\log{4}$. 
\end{proof}
\begin{remark}
The location of the poles must necessarily coincide with the location of the poles of the polynomial zeta functions and are thus at ${\rm Re}(s)=\frac{\log{3}}{\log{\lambda}}$.
\end{remark}
Then as in \cite{karl} we establish that the logarithm of the regularized determinant appears as a constant in the logarithm of the determinant of the discrete graph Laplacians.
\begin{cor}For the double $pq$-model we have that
\begin{equation} {\label{eq:doublepq}}
\log{\det\Delta_n}=|V_n|(\log{2}+\frac{\log{(pq)}}{2})+n\log{\frac{(1-q^2)(1-p^2)}{(pq)^2}}+\log{\det\mathcal{L}}
\end{equation}
\end{cor}

\begin{proof}
First of all, it is easy to calculate that $|V_n|=2\cdot 3^n$. Then, we use \eqref{eq:discretedet} with $P_d=\frac{1}{4pq}$, $	Q_0=1$ and $\alpha=-4$, $\alpha_n=1$, $\beta_1, \beta_2=-2\pm 2q$, $\beta_3, \beta_4 =-2 \pm 2p$ and $mult_n\beta_i=1$ for $n \geq 1$ and we get that
$$\det\Delta_n=4(2-2q^2)^{n}(-4pq)^{ \sum_{k=0}^{n-1} (3^k-1)}(2-2p^2)^{n}(-4pq)^{\sum_{k=0}^{n-1} (3^k-1)}$$
and thus
$$\det\Delta_n=2^{2\cdot 3^n}(1-q^2)^{n}(1-p^2)^{n}(pq)^{(3^n-2n-1)}$$
which gives us that
$$\log{\det\Delta_n}=2\cdot 3^n \log{2}+n(\log{(1-q^2)}+\log{(1-p^2)})+(3^n-2n-1)\log{(pq)}.$$
Since $\log{\det\mathcal{L}}=-\log{pq}$ we obtain our result. 
\end{proof}
We can observe that for $p=q=\frac{1}{2}$ this corresponds to the standard combinatorial graph Laplacian on the cyclic graph $C_{2\cdot 3^n}$. Therefore \eqref{eq:doublepq} becomes 
$$\log{\det\Delta_n}=n\log{9}+\log{4}$$
which is exactly as expected by observing that the cyclic graph has as many spanning trees as number of vertices and using Kirchhoff's Matrix-Tree theorem. This is also equivalent to the formula 
$$\prod_{k=1}^{2\cdot 3^n}(2-2\cos{\frac{2k\pi}{2\cdot 3^n}})=4\cdot 3^{2n}.$$
\section*{Acknowledgments}
We thank Professors Robert S. Strichartz, Gerald Dunne and Peter Grabner for helpful discussions and Anders Karlsson for suggesting the problem. The last-named author would also like to thank the Mathematics Department at the University of Connecticut for the hospitality during his research stay. Research of the first named author is supported by the Simons Foundation (via a Collaboration Grant for Mathematicians \#523544). Research of the second-named author is supported in part by NSF grant DMS-1613025.

\nocite{*}
\bibliographystyle{abbrv}
\bibliography{Biblio}

\end{document}